\documentclass[11pt,a4paper]{amsart}
\usepackage{a4wide}
\usepackage[T1]{fontenc}
\usepackage[utf8]{inputenc}
\usepackage{amsmath,amssymb,mathtools}
\usepackage{dsfont}
\usepackage{enumitem}
\usepackage{url}
\usepackage[colorlinks=true,linkcolor=blue,citecolor=blue,urlcolor=blue]{hyperref}
\newtheorem{theorem}{Theorem}[section]
\newtheorem{lemma}[theorem]{Lemma}
\newtheorem{proposition}[theorem]{Proposition}
\newtheorem{corollary}[theorem]{Corollary}

\theoremstyle{definition}

\newtheorem{problem}[theorem]{Problem}
\theoremstyle{remark}
\newtheorem{remark}[theorem]{Remark}

\theoremstyle{plain}
\newtheorem{maintheorem}{Theorem}

\newcommand{\clop}{\operatorname{Clop}}

\newcommand{\onto}{\twoheadrightarrow}

\newcommand{\fa}{\mathcal F}

\newcommand{\ptop}{\operatorname{p}}
\newcommand{\indicator}{\mathds{1}}

\title[Talagrand compacta, 2DCP, and pointwise quotients]{Talagrand compacta, 2DCP,\\ and pointwise quotients}

\author[T.~Kania]{Tomasz Kania\textsuperscript{*}}
\address[T.~Kania]{Mathematical Institute\\Czech Academy of Sciences\\\v{Z}itn\'a 25 \\115 67 Praha 1\\Czech Republic and Institute of Mathematics and Computer Science\\ Jagiellonian University\\ {\L}ojasiewicza 6, 30-348 Krak\'{o}w, Poland}
\email{kania@math.cas.cz, tomasz.marcin.kania@gmail.com}

\thanks{\textsuperscript{*}Corresponding author.}

\author[J.~K\k{a}kol]{Jerzy K\k{a}kol}
\address[J.~K\k{a}kol]{Faculty of Mathematics and Informatics, A. Mickiewicz University,
61-614 Pozna\'{n}, Poland}
\email{kakol@amu.edu.pl}

\date{Draft, July 2026}

\begin{document}

\begin{abstract}
We revisit Talagrand's CH compactum as a test object for the two-disjoint-copies property and for pointwise quotient questions.  The two-disjoint-copies property, or 2DCP, is a topological sufficient condition for the existence of infinite-dimensional metrisable quotients of spaces \(C_{\ptop}(X)\); recent work asks whether Talagrand's compactum has this property.  Assuming \(\diamondsuit(S)\) for a stationary co-stationary \(S\subseteq\omega_1\), we carry out Talagrand's inverse-limit construction with additional diagonalisation.  The resulting compactum \(T\) keeps Talagrand's conclusions: \(C(T)\) is Grothendieck, the weak-star compact ball \(M_1(T)\) contains no copy of \(\beta\omega\), and \(T\) has no non-trivial convergent sequences.  At the same time, no two disjoint non-metrisable closed subspaces of \(T\) are homeomorphic; hence \(T\) has no 2DCP and is not locally homogeneous.  We also give a ZFC example of a perfect compact space with 2DCP which is not locally homogeneous and contains neither \(\beta\omega\) nor \(2^\omega\).  Finally, for every Tychonoff space \(X\), we show that a continuous linear surjection \(C_{\ptop}(X)\to(\ell_p)_{\ptop}\), \(1\leqslant p<\infty\), yields a norm-one weak-star null sequence of finitely supported signed measures.  Thus \(C_{\ptop}(X)\) has the Josefson--Nissenzweig property and a quotient isomorphic to \((c_0)_{\ptop}\).  Consequently Talagrand compacta have no classical pointwise sequence quotients \((c_0)_{\ptop}\), \((\ell_p)_{\ptop}\), or \((\ell_\infty)_{\ptop}\).  The full metrisable quotient problem for these \(C_{\ptop}\)-spaces remains open.  Several open problems are included.
\end{abstract}


\maketitle

\section{Introduction}

Let \(X\) be a Tychonoff space.  Following Banakh, the second-named author, and \'{S}liwa, and the recent systematic study of the second-named author, Kurka, and \'{S}liwa, we say that \(X\) has the \emph{two-disjoint-copies property}, abbreviated 2DCP, if there is a sequence \((K_n)_{n\in\omega}\) of non-empty compact subsets of \(X\) such that each \(K_n\) contains two disjoint subsets homeomorphic to \(K_{n+1}\).  The definition is topological, but its motivation is the separable quotient problem for spaces of continuous functions with the pointwise topology.  The second-named author, Kurka, and \'{S}liwa recall the question whether, for every infinite compact space \(X\), the locally convex space \(C_{\ptop}(X)\) admits an infinite-dimensional metrisable, hence separable, quotient \cite[Problem 1]{KKS2026}.  Banakh, the second-named author, and \'{S}liwa proved that 2DCP is a sufficient compact-space criterion: if \(X\) has 2DCP, then \(C_{\ptop}(X)\) has an infinite-dimensional metrisable quotient \cite[Theorem 2]{BKS2018}; see also \cite[Theorem 3]{KKS2026}.

This makes 2DCP a useful bridge between internal self-similarity of compact spaces and quotient structure of \(C_{\ptop}(X)\).  It is strong enough to force a positive quotient theorem, but weak enough to hold in many non-homogeneous situations.  The difficult cases occur near Efimov-type compacta.  If a compact space contains a copy of the Cantor set or a copy of \(\beta\omega\), then it has 2DCP, whereas scattered compacta fail 2DCP \cite{KKS2026}.  Thus the genuinely delicate examples are compacta with no non-trivial convergent sequences and no copy of \(\beta\omega\).  This part of the landscape also includes the Sacks-model constructions of Brech and the related preservation theorem of Sobota and Zdomskyy for the Nikodym property in side-by-side Sacks extensions; the recent 2DCP work emphasises that the corresponding Brech and Sobota--Zdomskyy compacta of Efimov type have 2DCP \cite{Brech2006,SobotaZdomskyy2017,KKS2026}.  As complementary information in the same circle of ideas, Manev, Sobota, and Zdomskyy recently constructed, consistently with \(\omega_1<\mathfrak c\), a compact space \(K\) such that \(C(K)\) is Grothendieck, has density \(\omega_1\), and admits no equivalent norm which is strictly convex or sequentially Kadets--Klee \cite{ManevSobotaZdomskyy2026}.

A related elementary test question asks whether there is a perfect compact space with 2DCP which is not locally homogeneous and still contains neither of the two standard self-similar compacta \(2^\omega\) and \(\beta\omega\).  Proposition~\ref{prop:perfect-2dcp-not-local} gives a positive ZFC answer.  This shows that 2DCP itself is not a disguised homogeneity condition; local homogeneity is only one of several sufficient mechanisms for 2DCP.

Talagrand's celebrated CH construction \cite{Talagrand1980} lies exactly in this region.  It produces a compact space \(K\) such that \(C(K)\) has the Grothendieck property, while the weak-star compact unit ball of \(C(K)^*\) contains no copy of \(\beta\omega\).  For background on Grothendieck Banach spaces we refer to the survey of Gonz\'alez and the first-named author \cite{GonzalezKania2021}.  In particular, Talagrand's compactum contains neither a non-trivial convergent sequence nor a copy of \(\beta\omega\).  The second-named author, Kurka, and \'{S}liwa therefore ask explicitly whether Talagrand's compactum has 2DCP \cite[Problem 25]{KKS2026}.  The result below gives a negative answer in a relative-consistency sense: under Jensen's prediction principle, Talagrand's construction may be realised so as to fail 2DCP.  This gives a negative answer to the Talagrand-compactum version of Problem~25 for a diamond-guided realisation of Talagrand's construction; what is not claimed is that every possible realisation has the same behaviour.

There is an important terminological caveat.  Talagrand's compactum is produced by a transfinite construction scheme, not by a canonical unique object.  The construction depends on enumerations of disjoint clopen sequences, enumerations of signed measure sequences, a scheduling map, and the choices made at successor stages.  The theorem below should therefore be read as a statement about a carefully chosen realisation of Talagrand's scheme.  For a stationary set \(S\subseteq\omega_1\), \(\diamondsuit(S)\) denotes the usual Jensen principle: there are sets \(D_\alpha\subseteq\alpha\), \(\alpha\in S\), such that for every \(A\subseteq\omega_1\) the set of \(\alpha\in S\) with \(D_\alpha=A\cap\alpha\) is stationary.  For set-theoretic background on Jensen's diamond principle and its relatives, see Jensen's original work and Rinot's survey \cite{Jensen1972,Rinot2011}.

\begin{maintheorem}\label{thm:A}
Assume \(\diamondsuit(S)\) for some stationary co-stationary set \(S\subseteq\omega_1\).  There is a realisation of Talagrand's construction, with final compactum \(T\), such that no two disjoint non-metrisable closed subspaces of \(T\) are homeomorphic.  Consequently \(T\) does not have 2DCP.
\end{maintheorem}

The construction is still an honest Talagrand construction.  The prediction stages do not replace Talagrand's signed-measure splitting; they are inserted into the bookkeeping in a way which preserves the induction conditions \((A)\)--\((G)\).  The compatibility checks in Section~\ref{sec:compatibility} make explicit why Talagrand's theorem still applies.

\begin{maintheorem}\label{thm:B}
For \(T\) as in Theorem~\ref{thm:A}, the Banach space \(C(T)\) has the Grothendieck property and the weak-star compact unit ball of \(C(T)^*\) contains no copy of \(\beta\omega\).  In particular, \(T\) contains neither a non-trivial convergent sequence nor a copy of \(\beta\omega\).
\end{maintheorem}

\begin{corollary}\label{cor:not-local-homogeneous}
The compactum \(T\) of Theorem~\ref{thm:A} is not locally homogeneous.  It is not h-homogeneous either.
\end{corollary}

\begin{proof}
The second-named author, Kurka, and \'{S}liwa prove that every infinite compact locally homogeneous space has 2DCP \cite[Theorem 6]{KKS2026}.  Thus Theorem~\ref{thm:A} excludes local homogeneity.  Since \(T\) is zero-dimensional and infinite, h-homogeneity would give two disjoint non-empty clopen subspaces homeomorphic to \(T\), and hence 2DCP; again this is impossible by Theorem~\ref{thm:A}.
\end{proof}

\begin{corollary}\label{cor:efimov-consistency}
Assuming \(\diamondsuit(S)\) for some stationary co-stationary \(S\subseteq\omega_1\), there is an Efimov compactum which fails 2DCP and is not locally homogeneous.  Consequently
\[
        \operatorname{Con}(\mathrm{ZFC})
        \Longrightarrow
        \operatorname{Con}\bigl(\mathrm{ZFC}+\text{``some Talagrand compactum fails 2DCP''}\bigr).
\]
\end{corollary}

\begin{proof}
By Theorem~\ref{thm:B}, the compactum \(T\) of Theorem~\ref{thm:A} has no non-trivial convergent sequence and contains no copy of \(\beta\omega\); hence it is an Efimov compactum.  Theorem~\ref{thm:A} gives failure of 2DCP, and Corollary~\ref{cor:not-local-homogeneous} gives failure of local homogeneity.  Jensen's theorem that \(V=L\) implies \(\diamondsuit(S)\) for every stationary \(S\subseteq\omega_1\) gives the displayed relative-consistency statement; see, for instance, \cite[Theorem~1.2]{Rinot2011}.
\end{proof}

\begin{proposition}\label{prop:perfect-2dcp-not-local}
There is, in ZFC, a perfect compact zero-dimensional space \(X\) with 2DCP which is not locally homogeneous and which contains no copy of \(\beta\omega\) and no copy of \(2^\omega\).
\end{proposition}

\begin{proof}
Let \(D\) be the endpoint-deleted double-arrow compactum used in \cite[Example 26]{KKS2026}, \emph{i.e.}, \(D=([0,1]\times\{0,1\})\setminus\{(0,0),(1,1)\}\) with the order topology inherited from the lexicographic order.  It is compact, perfect, zero-dimensional, and h-homogeneous.  Hence it has 2DCP, and it contains neither \(\beta\omega\) nor \(2^\omega\).  Let \(L\) be the perfect zero-dimensional compact space without 2DCP constructed by the second-named author, Kurka, and \'{S}liwa \cite[Theorem 7]{KKS2026}.  Since any copy of \(2^\omega\) or of \(\beta\omega\) has 2DCP, \(L\) contains neither of them.  Put
\[
        X=D\sqcup L .
\]
Then \(X\) is perfect, compact, and zero-dimensional.  The property 2DCP is upward hereditary with respect to containing a subspace with 2DCP; hence \(X\) has 2DCP because the clopen summand \(D\) has 2DCP.

The space \(X\) contains no copy of \(2^\omega\).  Indeed, if \(C\subseteq X\) were such a copy, the clopen partition \(X=D\sqcup L\) would split \(C\) into two clopen pieces; one non-empty piece would again contain a copy of \(2^\omega\) and would lie in a single summand.  The same argument applies to \(\beta\omega\): a two-piece clopen partition of \(\beta\omega\) has an infinite clopen part \(\beta A\), hence a copy of \(\beta\omega\).  This is impossible because neither \(D\) nor \(L\) contains the relevant compactum.

It remains to see that \(X\) is not locally homogeneous.  Choose \(d\in D\) and \(\ell\in L\).  If local homogeneity gave a homeomorphism \(\varphi:U\to V\) between open neighbourhoods with \(\varphi(d)=\ell\), then, since \(D\) and \(L\) are clopen, we could shrink \(U\) so that \(U\subseteq D\) and \(\varphi[U]\subseteq L\).  Choose a non-empty clopen set \(C\subseteq D\) with \(d\in C\subseteq U\).  By h-homogeneity of \(D\), \(C\cong D\), so \(C\) has 2DCP.  Its image \(\varphi[C]\) is a compact subspace of \(L\) with 2DCP; by the same upward heredity, \(L\) itself has 2DCP, contradicting the choice of \(L\).
\end{proof}

The example above answers the natural homogeneity variant of the 2DCP question, but it is deliberately elementary.  The Talagrand construction considered in the rest of the paper is different in spirit: it remains in the Efimov region and is meant to serve as a test compactum for the pointwise quotient problem.  A secondary aim of the exposition is to isolate the small number of mechanisms in Talagrand's paper that are needed here, in the hope that the construction becomes more accessible to readers approaching it from \(C_{\ptop}\)-theory.

The pointwise quotient part of the paper contains a theorem which is not Talagrand-specific, but is specific to spaces of continuous functions.  If \(Y\subseteq\mathbb R^\omega\) is a Banach sequence space, we write \(Y_{\ptop}\) for \(Y\) endowed with the topology inherited from \(\mathbb R^\omega\).  Throughout, a quotient map between locally convex spaces means a continuous open linear surjection onto the target; saying that a space has a quotient isomorphic to \(F\) means that such a quotient is linearly homeomorphic to \(F\).  The theorem below in fact assumes only a continuous linear surjection.

\begin{maintheorem}\label{thm:cp-lp-quotient}
Let \(X\) be a Tychonoff space and let \(1\leqslant p<\infty\).  If there is a continuous linear surjection
\[
        Q:C_{\ptop}(X)\longrightarrow(\ell_p)_{\ptop},
\]
then \(C_{\ptop}(X)\) has the Josefson--Nissenzweig property.  Consequently, \(C_{\ptop}(X)\) has a quotient isomorphic to \((c_0)_{\ptop}\).
\end{maintheorem}

For Talagrand compacta this \(C_{\ptop}\)-specific theorem combines with a closed graph observation and Talagrand's original Banach-space conclusion.

\begin{corollary}\label{cor:talagrand-classical-quotients}
Let \(K\) be any Talagrand compactum satisfying Talagrand's conclusion.  Then \(C_{\ptop}(K)\) has no quotient isomorphic to \((c_0)_{\ptop}\), to \((\ell_\infty)_{\ptop}\), or to \((\ell_p)_{\ptop}\) for any \(1\leqslant p<\infty\).
\end{corollary}

This should be compared with the Banach-space situation.  The Rosenthal--Lacey theorem says that \(C(K)\), for infinite compact \(K\), has a quotient isomorphic to \(c_0\) or to \(\ell_2\) \cite{Rosenthal1969,Lacey1972}.  Since quotients of Grothendieck spaces are Grothendieck, the \(c_0\)-alternative is impossible for Talagrand's \(C(K)\).  Hence \(C(K)\) has a Banach quotient isomorphic to \(\ell_2\).  Corollary~\ref{cor:talagrand-classical-quotients} says that this quotient cannot be realised by a quotient map \(C_{\ptop}(K)\to (\ell_2)_{\ptop}\), and indeed rules out all classical pointwise sequence quotients \((\ell_p)_{\ptop}\) for Talagrand compacta.

The preceding consistency statement should be read as a statement about the topological sufficient criterion for quotients of \(C_{\ptop}(X)\), not as a negative solution of the separable quotient problem itself: the conclusion says that, consistently, Talagrand's compactum cannot be handled by the 2DCP criterion from \cite{BKS2018,KKS2026}.

The local mechanism is the following anti-extension lemma.  It is useful to phrase it in Boolean language.  If \(F\) is a closed subspace of a zero-dimensional compactum \(K\), put
\[
        I_F=\{a\in\clop(K): a\cap F=\varnothing\}.
\]
A homeomorphism \(h:F\to G\) is dual to a Boolean isomorphism
\[
        \clop(K)/I_G\cong \clop(K)/I_F.
\]
The successor step below kills such an isomorphism in a way which cannot be repaired by later splitting.

\begin{maintheorem}\label{thm:D}
Let \(\alpha<\omega_1\) be a countable stage of Talagrand's construction, and suppose that a diamond guess at stage \(\alpha\) gives disjoint infinite closed sets \(F,G\subseteq K_\alpha\) and a homeomorphism \(h:F\to G\).  Then the successor split
\[
        K_{\alpha+1}=(Y_\alpha\times\{0\})\cup (Z_\alpha\times\{1\})
\]
can be chosen so that Talagrand's conditions \((A)\)--\((G)\) remain valid, but no final homeomorphism \(H:\widehat F\to\widehat G\) extending \(h\) through \(\pi_\alpha\) can exist; that is, there is no such \(H\) satisfying
\[
        \pi_\alpha H=h\pi_\alpha .
\]
\end{maintheorem}

The proof of Theorem~\ref{thm:D} is the main point of the note.  A one-step fibre-cardinality diagonalisation is insufficient: a final homeomorphism may defer the information needed to later coordinates.  The repair is to combine Talagrand's condition \((G)\), which creates a parity split on a target sequence, with condition \((F)\), which gives a non-ultrafilter convergence filter on a paired source sequence.  The target split produces a final clopen set whose pullback under any alleged homeomorphism has alternating values on the source sequence.  Condition \((F)\) says that no final clopen set can alternate in this way.  This is why the obstruction is persistent.

\section{Talagrand's scheme and the splitter property}

We recall only the parts of Talagrand's construction that are needed below.  Let \(\Omega=\omega_1\).  Talagrand constructs compact spaces
\[
        K_\alpha\subseteq 2^\alpha\qquad (\omega\leqslant \alpha\leqslant \Omega)
\]
with bonding maps given by coordinate projections.  One starts with \(K_\omega=2^\omega\).  At a successor stage,
\begin{equation}\label{eq:successor}
        K_{\alpha+1}=(Y_\alpha\times\{0\})\cup (Z_\alpha\times\{1\}),
\end{equation}
where \(Y_\alpha,Z_\alpha\) are closed subsets of \(K_\alpha\) with \(Y_\alpha\cup Z_\alpha=K_\alpha\).  At a limit stage \(\rho\), the space \(K_\rho\) is the inverse limit of the earlier stages.  Since each \(\alpha<\omega_1\) is countable, every proper stage \(K_\alpha\) is compact metrisable and zero-dimensional.

For a clopen sequence scheduled at a stage \(\theta\), Talagrand writes
\[
        V_n(\theta,\alpha)=\pi_{\beta,\alpha}^{-1}(U_n(\beta,\gamma))\cap K_\alpha,
        \qquad \eta(\theta)=(\beta,\gamma),
\]
where \((U_n(\beta,\gamma))\) is a pairwise disjoint sequence of non-empty clopen subsets of \(K_\beta\).  For a signed measure sequence \((\nu_n(\theta,\delta))\subseteq M(K_\theta)\), define
\[
        B_n(\theta,\delta,\alpha)=
        \{\lambda\in M(K_\alpha):\|\lambda\|=1,
        (\pi_{\theta,\alpha})_*\lambda=\nu_n(\theta,\delta)\}.
\]
The enumerated sequences satisfy
\begin{equation}\label{eq:admissible}
        \|\nu_n(\theta,\delta)\|=|\nu_n(\theta,\delta)|(V_n(\theta,\theta))=1,
        \qquad \|\nu_n(\theta,\delta)^+\|\geqslant \frac12.
\end{equation}

Talagrand's induction constructs sets
\[
        I(\theta,\alpha),\quad J(\theta,\delta,\alpha),\quad
        N_0(\theta,\delta),N_1(\theta,\delta)\subseteq\omega
\]
satisfying conditions \((A)\)--\((G)\).  The precise formulation is in \cite[Section III]{Talagrand1980}; the two conditions used most directly here are:

\begin{itemize}[leftmargin=2em]
\item[\((F)\)] whenever \(\lambda_n\in B_n(\theta,\delta,\alpha)\), the weak-star limit
\[
        \lim_{n\to J(\theta,\delta,\alpha)}\lambda_n
\]
exists;

\item[\((G)\)] if \(\alpha=\alpha'+1\) and \(\eta(\alpha')=(\theta,\delta)\), then for every choice \(\lambda_n\in B_n(\theta,\delta,\alpha)\), both sets
\[
        \{n:\lambda_n(Z_{\alpha'}\times\{1\})\geqslant 1/6\},
        \qquad
        \{n:\lambda_n(Z_{\alpha'}\times\{1\})=0\}
\]
are infinite.
\end{itemize}

Conditions \((B)\)--\((D)\) imply that, for each pair \((\theta,\delta)\), the filter generated by the sets \(J(\theta,\delta,\alpha)\), \(\alpha>\max\{\theta,\delta\}\), together with the Fr\'echet filter, is not an ultrafilter: every member of it meets both \(N_0(\theta,\delta)\) and \(N_1(\theta,\delta)\) infinitely.

Two standing facts about the inverse system will be used repeatedly.  First, all bonding maps \(\pi_{\beta,\alpha}\colon K_\alpha\to K_\beta\), and hence all limit projections \(\pi_\alpha:=\pi_{\alpha,\omega_1}\colon T\to K_\alpha\), are surjective: at a successor stage this is the covering requirement \(Y_\alpha\cup Z_\alpha=K_\alpha\) in \eqref{eq:successor}, and surjectivity of the limit projections of an inverse system of non-empty compacta with surjective bonding maps is standard.  Secondly, if \(\rho\) is a limit ordinal and \(W\subseteq K_\rho\) is clopen, then, by compactness, \(W\) is a finite union of relative cylinders and hence depends on a finite set of coordinates; consequently there is \(\beta<\rho\) such that \(W=\pi_{\beta,\rho}^{-1}(W^\beta)\), where \(W^\beta=\pi_{\beta,\rho}[W]\) is the (uniquely determined) clopen trace of \(W\) in \(K_\beta\), and then the same holds at every stage of \([\beta,\rho)\).  In particular both facts apply to \(\rho=\omega_1\) and \(T=K_{\omega_1}\).

We shall use a slightly stronger scheduling convention.  If \(\eta(\alpha)=(\theta,\delta)\), then \(\theta\leqslant\alpha\) and \(\delta\leqslant\alpha\), and every pair is scheduled cofinally often above its two coordinates: for every \((\theta,\delta)\in[\omega,\omega_1)\times\omega_1\), the set
\[
        \{\alpha<\omega_1:\alpha\geqslant\max\{\theta,\delta\},\ \eta(\alpha)=(\theta,\delta)\}
\]
is cofinal in \(\omega_1\).  Talagrand's proof uses only eventual treatment of the admissible pairs, so replacing the original bijective schedule by such a surjective cofinal schedule is harmless.  The modified scheduling map used in Section~\ref{sec:predicted} is chosen with this convention, and the splitter lemma below applies to that construction.

We shall use the following direct extraction from Talagrand's proof.  It isolates exactly what the completed bookkeeping certifies about the final compactum; as the proof shows, the only induction condition consumed is \((G)\), together with the definitional layer of the enumerations.

\begin{lemma}\label{lem:splitter}
Let \(T\) be the final compactum of a realisation satisfying Talagrand's conditions.  Let \((W_n)\) be pairwise disjoint clopen subsets of \(T\), and let \(\lambda_n\in M(T)\) satisfy
\[
        \|\lambda_n\|=|\lambda_n|(W_n)=1,
        \qquad \|\lambda_n^+\|\geqslant \frac12 .
\]
Then there is a clopen set \(Z\subseteq T\) such that both sets
\[
        \{n:\lambda_n(Z)\geqslant 1/6\},
        \qquad
        \{n:\lambda_n(Z)=0\}
\]
are infinite.
\end{lemma}

\begin{proof}
The proof runs in four steps: the supports are reflected to a single countable stage; the total variations are reflected to a possibly larger countable stage; the projected sequence is recognised as one of the scheduled admissible sequences; and condition \((G)\) is applied at the successor stage where that sequence is treated.  The second step is the one that a naive reflection argument omits: an arbitrary countable projection need not preserve the masses \(\|\lambda_n\|\), and without mass preservation the projected sequence need not be admissible.

\emph{Step 1: reflection of the supports.}
Each \(W_n\) is non-empty, since \(|\lambda_n|(W_n)=1\).  By the remarks preceding the lemma, for each \(n\) there is \(\beta_n<\omega_1\) such that \(W_n\) is the full \(\pi_{\beta_n}\)-preimage of its clopen trace in \(K_{\beta_n}\).  As \(\operatorname{cf}(\omega_1)>\omega\), the ordinal \(\beta^0=\sup_n\beta_n\) is countable, and every \(W_n\) is the full preimage of a non-empty clopen trace at every stage of \([\beta^0,\omega_1)\).  Moreover, at any such stage the traces are pairwise disjoint: if a point \(s\) lay in the traces of \(W_m\) and \(W_n\) with \(m\neq n\), then, by surjectivity of the limit projection, any \(t\in T\) projecting onto \(s\) would lie in \(W_m\cap W_n=\varnothing\).

\emph{Step 2: reflection of the masses.}
We claim that there is \(\beta\in[\beta^0,\omega_1)\) such that
\begin{equation}\label{eq:mass-reflection}
        \bigl\|(\pi_\beta)_*\lambda_n\bigr\|=1
        \qquad\text{for every }n\in\omega .
\end{equation}
Fix \(n\) and \(\varepsilon>0\).  By the Riesz description of the total variation and the Stone--Weierstrass theorem (locally constant functions are uniformly dense in \(C(T)\), since \(\clop(T)\) separates points), there are a finite clopen partition \((E_i)_{i\leqslant k}\) of \(T\) and scalars \(|c_i|\leqslant 1\) with \(\sum_{i\leqslant k}c_i\,\lambda_n(E_i)>1-\varepsilon\); in particular \(\sum_{i\leqslant k}|\lambda_n(E_i)|>1-\varepsilon\).  The finitely many \(E_i\) are full preimages of clopen sets at a common countable stage \(\gamma\).  For any \(\beta'\geqslant\gamma\), the traces \(E_i^{\beta'}=\pi_{\beta'}[E_i]\) form a clopen partition of \(K_{\beta'}\) (pairwise disjointness again uses surjectivity of \(\pi_{\beta'}\)), and \((\pi_{\beta'})_*\lambda_n(E_i^{\beta'})=\lambda_n(E_i)\), whence
\[
        \bigl\|(\pi_{\beta'})_*\lambda_n\bigr\|
        \geqslant\sum_{i\leqslant k}|\lambda_n(E_i)|
        >1-\varepsilon .
\]
Applying this with \(\varepsilon=1/k\), \(k\geqslant 1\), and using \(\operatorname{cf}(\omega_1)>\omega\) twice --- once over \(k\) for fixed \(n\), and once over \(n\) --- produces \(\beta\in[\beta^0,\omega_1)\) satisfying \eqref{eq:mass-reflection}; the reverse inequality \(\|(\pi_\beta)_*\lambda_n\|\leqslant\|\lambda_n\|=1\) always holds, because a push-forward is a contraction for the total variation norm.  Note also that \(\beta'\mapsto\|(\pi_{\beta'})_*\lambda_n\|\) is non-decreasing, since \((\pi_{\beta})_*=(\pi_{\beta,\beta'})_*(\pi_{\beta'})_*\) for \(\beta\leqslant\beta'\); hence \eqref{eq:mass-reflection} persists at every stage \(\geqslant\beta\).

\emph{Step 3: the projected sequence is scheduled.}
Let \(W_n^\beta\) denote the traces of Step 1 at the stage \(\beta\) of Step 2.  They form a pairwise disjoint sequence of non-empty clopen subsets of the metrisable compactum \(K_\beta\); since the enumeration \((U_n(\beta,\gamma))_{n\in\omega,\,\gamma<\omega_1}\) lists all such sequences (\(\diamondsuit\) implies CH, so there are only \(\omega_1\) of them), there is \(\gamma<\omega_1\) with \(U_n(\beta,\gamma)=W_n^\beta\) for all \(n\).  Choose \(\theta\geqslant\max\{\beta,\gamma\}\) with \(\eta(\theta)=(\beta,\gamma)\); then by the definition of the \(V\)'s,
\[
        V_n(\theta,\theta)=\pi_{\beta,\theta}^{-1}\bigl(U_n(\beta,\gamma)\bigr)\cap K_\theta,
        \qquad
        W_n=\pi_\theta^{-1}\bigl(V_n(\theta,\theta)\bigr).
\]
Set \(\mu_n=(\pi_\theta)_*\lambda_n\in M(K_\theta)\).  We verify that the sequence \((\mu_n)\) is admissible in the sense of \eqref{eq:admissible}.

First, \(\|\mu_n\|=1\), by Step 2 and monotonicity, since \(\theta\geqslant\beta\).

Secondly, \(|\mu_n|\bigl(V_n(\theta,\theta)\bigr)=1\).  Indeed, if \(B\) is a Borel subset of \(K_\theta\setminus V_n(\theta,\theta)\), then \(\pi_\theta^{-1}(B)\subseteq T\setminus W_n\), which is \(|\lambda_n|\)-null; so \(\mu_n\) vanishes on every Borel subset of the clopen set \(K_\theta\setminus V_n(\theta,\theta)\), whence \(|\mu_n|\bigl(K_\theta\setminus V_n(\theta,\theta)\bigr)=0\).

Thirdly, \(\|\mu_n^+\|\geqslant 1/2\).  Put \(\sigma=(\pi_\theta)_*\lambda_n^{+}+(\pi_\theta)_*\lambda_n^{-}\).  From \(\mu_n=(\pi_\theta)_*\lambda_n^{+}-(\pi_\theta)_*\lambda_n^{-}\) we get \(|\mu_n|\leqslant\sigma\); both are positive measures of the same total mass, namely
\[
        \sigma(K_\theta)=\|\lambda_n^{+}\|+\|\lambda_n^{-}\|=\|\lambda_n\|=1=\|\mu_n\|=|\mu_n|(K_\theta),
\]
so \(|\mu_n|=\sigma\).  Consequently
\[
        \mu_n^{+}=\tfrac12\bigl(|\mu_n|+\mu_n\bigr)=(\pi_\theta)_*\lambda_n^{+},
        \qquad
        \|\mu_n^{+}\|=\|\lambda_n^{+}\|\geqslant\tfrac12 .
\]
In other words, once the projection preserves the total variation, it also commutes with the Jordan decomposition; this is exactly where \eqref{eq:mass-reflection} is needed, and it is the point that fails for an arbitrary countable stage.

By \eqref{eq:admissible} and the totality of the enumeration \((\nu_n(\theta,\delta))_{n\in\omega,\,\delta<\omega_1}\) (CH again), there is \(\delta<\omega_1\) with
\[
        \mu_n=\nu_n(\theta,\delta)\qquad\text{for all }n\in\omega .
\]

\emph{Step 4: application of condition \((G)\).}
Choose \(\alpha'\geqslant\max\{\theta,\delta\}\) with \(\eta(\alpha')=(\theta,\delta)\).  Put \(\widetilde\lambda_n=(\pi_{\alpha'+1})_*\lambda_n\in M(K_{\alpha'+1})\).  Then
\[
        1=\|\mu_n\|
        =\bigl\|(\pi_{\theta,\alpha'+1})_*\widetilde\lambda_n\bigr\|
        \leqslant\|\widetilde\lambda_n\|
        \leqslant\|\lambda_n\|=1,
\]
and \((\pi_{\theta,\alpha'+1})_*\widetilde\lambda_n=(\pi_\theta)_*\lambda_n=\nu_n(\theta,\delta)\).  Hence \(\widetilde\lambda_n\in B_n(\theta,\delta,\alpha'+1)\) for every \(n\), and condition \((G)\) at the successor rank \(\alpha'+1\) yields that both sets
\[
        \bigl\{n:\widetilde\lambda_n(Z_{\alpha'}\times\{1\})\geqslant 1/6\bigr\},
        \qquad
        \bigl\{n:\widetilde\lambda_n(Z_{\alpha'}\times\{1\})=0\bigr\}
\]
are infinite.  By \eqref{eq:successor}, \(Z_{\alpha'}\times\{1\}=\{t\in K_{\alpha'+1}:t(\alpha')=1\}\) is clopen in \(K_{\alpha'+1}\) (the trace of a one-coordinate cylinder), so
\[
        Z:=\pi_{\alpha'+1}^{-1}\bigl(Z_{\alpha'}\times\{1\}\bigr)
        =\{t\in T:t(\alpha')=1\}
\]
is clopen in \(T\); and since \(Z\) is a full preimage, \(\lambda_n(Z)=\widetilde\lambda_n(Z_{\alpha'}\times\{1\})\) for every \(n\).  The two sets in the statement are therefore infinite for \((\lambda_n)\) and \(Z\), as required.
\end{proof}

\begin{remark}\label{rem:splitter}
Three by-products of the proof are worth recording.  First, the hypotheses of Lemma~\ref{lem:splitter} are inherited by subsequences --- a subsequence of a pairwise disjoint sequence of non-empty clopen sets is again one, and the corresponding measures remain admissible --- so the conclusion holds within every subsequence of \((\lambda_n)\); this is the form used in Corollary~\ref{cor:borel-splitter}.  Secondly, the splitting clopen set can always be taken to be a single successor-coordinate cylinder \(Z=\{t\in T:t(\alpha')=1\}\).  Thirdly, the proof consumes only condition \((G)\) and the totality of the enumerations and of the scheduling map; it is insensitive to how the remaining freedom in the construction is spent.  This is why the splitter survives the diamond-guided bookkeeping of Section~\ref{sec:predicted} unchanged, while condition \((F)\) is kept in reserve for Lemma~\ref{lem:source-filter}.
\end{remark}

\begin{corollary}\label{cor:borel-splitter}
Let \((\mu_n)\) be probability measures on \(T\) carried by pairwise disjoint Borel sets.  Then no subsequence of \((\mu_n)\) is weak-star convergent.
\end{corollary}

\begin{proof}
By Rosenthal's disjointification lemma \cite{Rosenthal1970}, for every infinite set of indices and every \(\varepsilon>0\) one may pass to an infinite subset and find pairwise disjoint open sets \(U_n\) with \(\mu_n(U_n)>1-\varepsilon\).  By zero-dimensionality and regularity, choose pairwise disjoint clopen sets \(W_n\subseteq U_n\) with \(\mu_n(W_n)>1-2\varepsilon\).  Apply Lemma \ref{lem:splitter} to the normalised restrictions \(\mu_n(\,\cdot\cap W_n)/\mu_n(W_n)\).  For instance, with \(\varepsilon=1/100\), one obtains a clopen \(Z\) such that \(\mu_n(Z)\geqslant 49/300\) on infinitely many indices and \(\mu_n(Z)\leqslant 1/50\) on infinitely many others.  Since \(\indicator_Z\in C(T)\), weak-star convergence is impossible on that subsequence.
\end{proof}

\section{Reflection of final homeomorphisms}

We shall need a standard spectral observation for compact subspaces of \(2^{\omega_1}\).  If \(F\subseteq T\), write \(F_\alpha=\pi_\alpha[F]\).

\begin{lemma}\label{lem:reflection}
Let \(T\subseteq 2^{\omega_1}\) be an inverse limit as above.  Let \(F,G\subseteq T\) be disjoint closed subspaces and let \(H:F\to G\) be a homeomorphism.  Then there is a club \(C\subseteq\omega_1\) such that for every \(\alpha\in C\):
\begin{enumerate}[label=\textup{(\roman*)},leftmargin=2em]
\item \(F_\alpha\cap G_\alpha=\varnothing\);
\item there is a homeomorphism \(H_\alpha:F_\alpha\to G_\alpha\) satisfying
\[
        \pi_\alpha H=H_\alpha\pi_\alpha;
\]
\item if \(F\) is non-metrisable, then \(F_\alpha\) is infinite for all sufficiently large \(\alpha\).
\end{enumerate}
\end{lemma}

\begin{proof}
Since \(F\) and \(G\) are disjoint closed subsets of the zero-dimensional compactum \(T\), there is a clopen set \(A\subseteq T\) with \(F\subseteq A\) and \(A\cap G=\varnothing\).  The set \(A\) depends on finitely many coordinates, so (i) holds above some stage.

For the factorisation, consider one output coordinate \(\xi<\omega_1\).  The map \(x\mapsto H(x)(\xi)\) is a continuous \(\{0,1\}\)-valued function on \(F\).  Its \(1\)-set is clopen in \(F\); since \(F\) is closed in the zero-dimensional compactum \(T\), this relative clopen set extends to a clopen subset of \(T\).  Every clopen subset of \(T\subseteq2^{\omega_1}\) depends on finitely many coordinates, by compactness and the cylinder base.  Hence this coordinate of \(H\) depends on finitely many input coordinates.  Do the same for all \(\xi\) and also for \(H^{-1}\).  Closing a countable ordinal under all the resulting finite dependency maps gives a club \(C\).  At each \(\alpha\in C\), both \(H\) and \(H^{-1}\) factor through \(\pi_\alpha\), and the factored maps are inverse homeomorphisms.

Finally, if \(F_\alpha\) were finite cofinally often, then by monotonicity the finite cardinalities would stabilise on a tail; the inverse limit over that tail would be finite, contradicting non-metrisability of \(F\).  Equivalently, since \(F\ne\varnothing\), once some countable projection is infinite, all later projections are infinite; and if all countable projections were finite with unbounded finite sizes, a countable limit of the stages where the sizes increase would already have infinite projection.
\end{proof}

Fix once and for all an enumeration \((C_\xi)_{\xi<\omega_1}\) of the finite-coordinate clopen cylinders of \(2^{\omega_1}\), an enumeration of ordered pairs of such cylinders, and a recursive pairing scheme for three kinds of data: the trace of the first closed set, the trace of the second closed set, and the graph.  We arrange the enumerations so that there is a club \(C_{\mathrm{cyl}}\subseteq\omega_1\) with the following coherence property: for every \(\alpha\in C_{\mathrm{cyl}}\), the initial segments of the enumerations list exactly the cylinders, and the ordered pairs of cylinders, whose coordinate supports are contained in \(\alpha\).  This is obtained by closing \(\alpha\) under the fixed coding maps and their partial inverses.

For \(\alpha\in C_{\mathrm{cyl}}\) and \(\operatorname{supp}C_\xi\subseteq\alpha\), let \(C_\xi^\alpha\) denote the trace in \(K_\alpha\) of the finite condition defining \(C_\xi\).  The family of these traces is a canonical clopen base of \(K_\alpha\).  A closed set \(F\subseteq K_\alpha\) is coded by the indices \(\xi\) for which \(F\cap C_\xi^\alpha\neq\varnothing\); a closed graph in \(K_\alpha\times K_\alpha\) is coded analogously by the canonical product traces.  The decoding at stage \(\alpha\) is internal to the already constructed compactum \(K_\alpha\): a code is accepted only if it defines two disjoint infinite closed sets and a closed relation which is the graph of a homeomorphism between them.

\begin{lemma}\label{lem:canonical-code}
For every final triple \((F,G,H)\), where \(F,G\subseteq T\) are closed and \(H:F\to G\) is a homeomorphism, there are a global code \(A\subseteq\omega_1\) and a club \(C_{\mathrm{code}}\subseteq\omega_1\) such that, whenever \(\alpha\in C_{\mathrm{code}}\) and \(H\) reflects to a homeomorphism \(H_\alpha:F_\alpha\to G_\alpha\), the initial segment \(A\cap\alpha\) is exactly the canonical stage-\(\alpha\) code of \((F_\alpha,G_\alpha,H_\alpha)\).
\end{lemma}

\begin{proof}
Use the three reserved parts of the pairing scheme to form a single set \(A\subseteq\omega_1\): record whether each canonical cylinder meets \(F\), whether it meets \(G\), and whether each canonical product cylinder meets the graph of \(H\).  Let \(C_{\mathrm{code}}\) be the intersection of \(C_{\mathrm{cyl}}\) with the club of ordinals closed under the pairing maps and their partial inverses.

Fix \(\alpha\in C_{\mathrm{code}}\) at which \(H\) reflects.  For every cylinder whose support is contained in \(\alpha\), its trace meets \(F_\alpha\) if and only if its full inverse image meets \(F\); the same holds for \(G\), and, in the product, for the graph of \(H\) and the graph of \(H_\alpha\).  The coherence of the enumerations therefore implies that \(A\cap\alpha\) is exactly the canonical stage-\(\alpha\) code of \((F_\alpha,G_\alpha,H_\alpha)\).
\end{proof}

\begin{lemma}\label{lem:diamond-guess}
Assume \(\diamondsuit(S)\).  There is a diamond sequence on \(S\) which, for every final homeomorphism \(H:F\to G\) between disjoint closed subspaces of \(T\), guesses the trace \((F_\alpha,G_\alpha,H_\alpha)\) for stationarily many \(\alpha\) belonging to the club of Lemma \ref{lem:reflection}.
\end{lemma}

\begin{proof}
Let \(A\subseteq\omega_1\) and \(C_{\mathrm{code}}\) be given by Lemma~\ref{lem:canonical-code}.  Intersect \(C_{\mathrm{code}}\) with the reflection club of Lemma~\ref{lem:reflection}.  For every \(\alpha\) in this club, the initial segment \(A\cap\alpha\) decodes exactly the projected closed sets \(F_\alpha,G_\alpha\) and the projected closed graph of \(H_\alpha\).  By \(\diamondsuit(S)\), the set of \(\alpha\in S\) for which the diamond sequence guesses \(A\cap\alpha\) is stationary.  Its intersection with the combined club is non-empty, indeed stationary.
\end{proof}

\section{Compatibility with Talagrand's induction}\label{sec:compatibility}

For reference, we spell out the bookkeeping conditions from Talagrand's construction in the notation used in this paper.  This subsection is not a new argument; it is an audit trail showing exactly where the diagonal successor step differs from the original successor step and why the induction still closes.  The constants and inequalities in the finite requirements \((a)\)--\((h)\) below are copied from Talagrand's successor step; only the notation has been changed.

For \(\omega\leqslant\theta<\alpha\) and \(\delta<\alpha\), Talagrand constructs sets \(I(\theta,\alpha)\), \(J(\theta,\delta,
\alpha)\), and \(N_0(\theta,\delta),N_1(\theta,\delta)\subseteq\omega\) satisfying:
\begin{enumerate}[label=\textup{(\Alph*)},leftmargin=2.6em]
\item \(J(\theta,\delta,\alpha)\subseteq I(\theta,\alpha)\).
\item \(N_0(\theta,\delta)\cap N_1(\theta,\delta)=\varnothing\).
\item \(J(\theta,\delta,\alpha)\cap N_e(\theta,\delta)\) is infinite for \(e=0,1\).
\item If \(\alpha'\leqslant\alpha\) and \(\delta<\alpha'\), then
\[
        J(\theta,\delta,\alpha)\setminus J(\theta,\delta,\alpha')
        \quad\text{and}\quad
        I(\theta,\alpha)\setminus I(\theta,\alpha')
\]
are finite whenever both sides are defined.
\item If \(\omega\leqslant\theta\leqslant\alpha'<\alpha\), then for all sufficiently large \(n\in I(\theta,\alpha)\),
\[
        t,t'\in V_n(\theta,\alpha),\quad
        \pi_\theta(t)=\pi_\theta(t')
        \quad\Longrightarrow\quad
        t(\alpha')=t'(\alpha').
\]
\item For every choice \(\lambda_n\in B_n(\theta,\delta,
\alpha)\), the weak-star limit \(\lim_{n\to J(\theta,\delta,
\alpha)}\lambda_n\) exists.
\item If \(\alpha=\alpha'+1\) and \(\eta(\alpha')=(\theta,
\delta)\), then for every choice \(\lambda_n\in B_n(\theta,
\delta,
\alpha)\), both sets
\[
        \{n:\lambda_n(Z_{\alpha'}\times\{1\})\geqslant 1/6\},
        \qquad
        \{n:\lambda_n(Z_{\alpha'}\times\{1\})=0\}
\]
are infinite.
\end{enumerate}

At a successor stage \(\alpha+1\), write \(\eta(\alpha)=(\theta,
\delta^*)\).  Enumerate \([\omega,
\alpha]\) as \((\theta_i)_{i\in I}\) with \(\theta_0=\alpha\).  Let \(\rho:\omega\to\alpha\), \(q\mapsto\rho_q\), be a surjection such that every fibre meets the even and odd integers infinitely often.  Talagrand chooses clopen sets \(H_q\) and integers \(n_q,m_{i,q}\) satisfying the following finite requirements.  Here \(\varepsilon(q)=0\) for even \(q\), \(\varepsilon(q)=1\) for odd \(q\), \(P(q)\) is the finite set of coordinates generated by \(\rho_j\), \(j\leqslant q\), and by the clopen sets already chosen, and \(\mu^i_\delta\) denotes the relevant weak-star cluster point supplied by the induction.
\begin{enumerate}[label=\textup{(\alph*)},leftmargin=2.6em]
\item \(m_{i,q+1}>m_{i,q}\) for every \(i\in I\).
\item For every \(i\in I\setminus\{0\}\),
\[
        m_{i,q}\in J(\theta_i,\rho_q,\alpha)
        \cap N_{\varepsilon(q)}(\theta_i,\rho_q).
\]
\item The common variation-limit measure \(\nu\) for the scheduled target sequence satisfies
\[
        \nu\bigl(V_{m_{i,q}}(\theta_i,
\alpha)\bigr)
        \leqslant 2^{-i-q-7}.
\]
\item If \(t,t'\in V_{m_{i,q}}(\theta_i,
\alpha)\) and \(\pi_{\theta_i}(t)=\pi_{\theta_i}(t')\), then \(\pi_{P(q)}(t)=\pi_{P(q)}(t')\).
\item For every \(\lambda\in B_{m_{i,q}}(\theta_i,
\rho_q,
\alpha)\),
\[
        \bigl\|(\pi_{P(q)})_*\lambda-(\pi_{P(q)})_*\mu^i_{\rho_q}\bigr\|
        \leqslant 2^{-q}.
\]
\item \(n_q\) belongs to the selected infinite set \(I_0\), and \(n_{q+1}>n_q\).  In particular the integers \(n_q\) are pairwise distinct.
\item For every \(\lambda\in B_{n_q}(\theta,
\delta^*,
\alpha)\),
\[
        |\lambda|\bigl(V_{m_{i,j}}(\theta_i,
\alpha)\bigr)
        \leqslant 2^{-i-j-6}
        \qquad(i,j<q).
\]
\item If \(W_{n_q}\subseteq V_{n_q}(\theta,
\theta)\) is the clopen set with \(\nu_{n_q}(\theta,
\delta^*)(W_{n_q})>1/3\), then
\[
        H_q=\pi_{\theta,
\alpha}^{-1}(W_{n_q})
        \setminus\bigcup_{i,j<q}V_{m_{i,j}}(\theta_i,
\alpha).
\]
\end{enumerate}

Requirement \((b)\) is imposed only for \(i\neq0\).  The distinguished index \(m_{0,q}\) is chosen separately, exactly as in Talagrand's proof.  Requirements \((a)\) and \((c)\) hold for all sufficiently large choices; requirement \((d)\) is automatic because \(\theta_0=\alpha\); and requirement \((e)\) holds for infinitely many choices because
\[
        B_n(\alpha,\rho_q,\alpha)=\{\nu_n(\alpha,\rho_q)\}
\]
and \(\mu^0_{\rho_q}\) is chosen as a weak-star cluster point of the corresponding sequence.  Thus no undefined set \(J(\alpha,\rho_q,\alpha)\) or \(N_{\varepsilon(q)}(\alpha,\rho_q)\) is used.

The diagonal step below changes only one item: for the newly introduced source sequence it adds the strictly increasing indices \(r_q=n_q\), for \(q\) in one fibre of \(\rho\), to the new set \(I(\alpha,
\alpha+1)\).  The following table records the effect on Talagrand's induction.
\begin{center}
\begin{tabular}{p{0.25\textwidth}p{0.65\textwidth}}
\hline
Condition & Reason it survives at a diagonal stage\\
\hline
\((A)\)--\((D)\) for old pairs & The old sets \(J(\theta,
\delta,
\alpha+1)\) and \(I(\theta,
\alpha+1)\) are Talagrand's original ones.\\
New source pair \((\alpha,
\alpha)\) & There is no earlier value of \(I(\alpha,\cdot)\) or \(J(\alpha,
\alpha,
\cdot)\).  The sets \(N_0,N_1\) are chosen as the even and odd pieces of \(\{r_q\}\), so \((B)\) and \((C)\) hold immediately.\\
\((E)\) & For old indices this is Talagrand's proof.  For the added indices \(r_q\), the support \(R_{r_q}\) misses the boundary \(Y_\alpha\cap Z_\alpha\), so the new coordinate is constant on each fibre over \(R_{r_q}\); see Lemma~\ref{lem:parity}.\\
\((F)\) & Old pairs are unchanged.  Every norm-one lift of \(\delta_{x_{r_q}}\) is uniquely \(\delta_{(x_{r_q},0)}\); along the added set \(J(\alpha,\alpha,\alpha+1)=\{r_q\}\), these measures converge to \(\delta_{(x,0)}\).\\
\((G)\) & The target pair \((\alpha,0)\) is treated by Talagrand's ordinary estimate.  The new source pair is an admissible sequence and is scheduled later by the ordinary scheduler.\\
\hline
\end{tabular}
\end{center}

Thus the additional source sequence is compatible with the finite successor requirements.  The verification of \((A)\)--\((G)\) is not an appeal to an unchanged proof in the presence of hidden changes: the only changed objects are listed above, and their required properties are checked directly in Lemma~\ref{lem:paired-successor-step} and Lemma~\ref{lem:source-filter}.

\section{The persistent Boolean anti-extension lemma}\label{sec:anti-extension}

This section proves the local diagonalisation used at the diamond stages.  The proof is written in Talagrand's notation, but the idea is simple: force a clopen coordinate to alternate on the images \(h(x_n)\), while Talagrand's filter condition forbids any final clopen set from alternating on the corresponding source points \(x_n\).

Let \(\alpha>\omega\) be a stage at which \(K_\alpha\) is constructed.  Suppose we are given disjoint infinite closed sets \(F,G\subseteq K_\alpha\) and a homeomorphism
\[
        h:F\longrightarrow G.
\]
Choose a non-isolated point \(x\in F\), a sequence of distinct points \(x_n\in F\setminus\{x\}\) with \(x_n\to x\), and put
\[
        y=h(x),\qquad y_n=h(x_n).
\]
Then \(y_n\to y\), and the \(y_n\)'s are distinct.  Choose disjoint clopen neighbourhoods \(P,Q\subseteq K_\alpha\) of \(x\) and \(y\), respectively, and pass to a subsequence so that \(x_n\in P\), \(y_n\in Q\) for all \(n\).  Choose pairwise disjoint clopen sets
\[
        P_n\subseteq P,\qquad Q_n\subseteq Q,
\]
with \(x_n\in P_n\), \(y_n\in Q_n\), and with \(x,y\notin P_n\cup Q_n\).  Put
\[
        R_n=P_n\cup Q_n .
\]

At the bookkeeping level, reserve the target index \(0\) and the source index \(\alpha\).  Set
\[
        \eta(\alpha)=(\alpha,0),
        \qquad U_n(\alpha,0)=R_n,
\]
so that \(V_n(\alpha,\alpha)=R_n\).  Put
\[
        \nu_n(\alpha,0)=\delta_{y_n}
        \quad\text{and}\quad
        \nu_n(\alpha,\alpha)=\delta_{x_n}.
\]
Both sequences are admissible in the sense of \eqref{eq:admissible}.  In Talagrand's successor construction for \(K_{\alpha+1}\), take
\[
        W_n=Q_n\subseteq R_n;
\]
then \(\nu_n(\alpha,0)(W_n)=1\).  Let \((\rho_q)_{q\in\omega}\) be the usual surjection from \(\omega\) onto \(\alpha\), with each fibre meeting the even and odd integers infinitely often.  Let
\[
        Q^*=\{q\in\omega:\rho_q=0\}.
\]

\begin{lemma}\label{lem:paired-successor-step}
The successor construction can be carried out so that Talagrand's finite requirements \((a)\)--\((h)\) of Section~\ref{sec:compatibility} hold.  The integers \(n_q\) are strictly increasing.  If, for \(q\in Q^*\), we put
\[
        r_q=n_q,
\]
then one may additionally require
\[
        \{r_q:q\in Q^*\}\subseteq I(\alpha,\alpha+1)
\]
and set
\[
        J(\alpha,\alpha,\alpha+1)=\{r_q:q\in Q^*\}
\]
for the source sequence \(\nu_n(\alpha,\alpha)=\delta_{x_n}\), while Talagrand's conditions \((A)\)--\((G)\) remain valid.
\end{lemma}

\begin{proof}
Choose the integers \(n_q\) and \(m_{i,q}\) by the finite induction described in Section~\ref{sec:compatibility}.  The added source sequence imposes no requirement on those choices: requirements \((a)\)--\((h)\) concern the target sequence \(\nu_n(\alpha,0)=\delta_{y_n}\).  At each finite step, \(n_q\) is chosen from the same infinite set as in Talagrand's proof.  For \(i>0\), the indices \(m_{i,q}\) are chosen from the same cofinal subsets of
\[
        J(\theta_i,\rho_q,\alpha)
        \cap N_{\varepsilon(q)}(\theta_i,\rho_q)
\]
while satisfying the remaining finite requirements.  The distinguished index \(m_{0,q}\) is chosen separately, again exactly as in Talagrand's proof: requirements \((a)\) and \((c)\) hold for sufficiently large choices, \((d)\) is automatic because \(\theta_0=\alpha\), and \((e)\) holds for infinitely many indices because \(B_n(\alpha,\rho_q,\alpha)=\{\nu_n(\alpha,\rho_q)\}\) and \(\mu^0_{\rho_q}\) is a cluster point of that sequence.

After these choices are made, we merely record the already chosen subsequence \(r_q=n_q\) on the fibre \(Q^*\) as the source \(J\)-set.  Requirement \((f)\) gives \(n_{q+1}>n_q\), hence the indices \(r_q\) are pairwise distinct.  In the present target case \(B_n(\alpha,0,\alpha)=\{\delta_{y_n}\}\), and condition \((g)\) says precisely that, when \(n_q\) is chosen, the point \(y_{n_q}\) avoids all the finitely many previously selected supports \(V_{m_{i,j}}(\theta_i,\alpha)\), \(i,j<q\).  Hence \(y_{n_q}\in H_q\) below.

The only extra move is to enlarge \(I(\alpha,\alpha+1)\) by the set \(\{r_q:q\in Q^*\}\).  This does not affect the old pairs \((\alpha,\delta)\), \(\delta<\alpha\), because their sets \(J(\alpha,\delta,\alpha+1)\) are still the original Talagrand sets \(\{m_{0,q}:\rho_q=\delta\}\).  For the new pair \((\alpha,\alpha)\) there is no previous value of \(I(\alpha,\cdot)\), so condition \((D)\) imposes no backwards compatibility at this stage.

It remains to check condition \((E)\) for the added indices \(r_q\).  For those indices
\[
        V_{r_q}(\alpha,\alpha)=R_{r_q}=P_{r_q}\cup Q_{r_q}.
\]
As the sets \(P_n\) and \(Q_n\) are pairwise disjoint and clopen, and as each \(H_s\) is contained in \(Q_{n_s}\), we shall have
\[
        Z_\alpha\cap R_{r_q}=H_q\quad(q\in Q^*\text{ even}),
        \qquad
        Z_\alpha\cap R_{r_q}=\varnothing\quad(q\in Q^*\text{ odd}).
\]
In the even case \(H_q\) is clopen in \(K_\alpha\), hence no point of \(R_{r_q}\) lies in the boundary \(Y_\alpha\cap Z_\alpha\); in the odd case this is immediate.  Therefore the new coordinate is constant on each fibre of the projection over \(R_{r_q}\).  Since equality of the \(K_\alpha\)-projection already fixes all old coordinates, condition \((E)\) holds for the added indices.

For the source sequence \(\nu_n(\alpha,\alpha)=\delta_{x_n}\), the chosen set
\[
        J(\alpha,\alpha,\alpha+1)=\{r_q:q\in Q^*\}
\]
is legitimate for condition \((F)\).  Indeed all \(H_s\)'s lie in the clopen set \(Q\), while \(x\) and all \(x_n\)'s lie in the disjoint clopen set \(P\).  Hence the fibre of \(K_{\alpha+1}\to K_\alpha\) over each \(x_n\), and over \(x\), is a singleton on the \(0\)-side.

We also have to check uniqueness for arbitrary signed lifts.  Let
\[
        \lambda\in B_{r_q}(\alpha,\alpha,\alpha+1).
\]
Then
\[
        \bigl\|(\pi_{\alpha,\alpha+1})_*\lambda\bigr\|
        =\|\delta_{x_{r_q}}\|=1=\|\lambda\|.
\]
Since \(|(\pi_{\alpha,\alpha+1})_*\lambda|\leqslant(\pi_{\alpha,\alpha+1})_*|\lambda|\) and both positive measures have total mass one, they are equal.  Thus
\[
        (\pi_{\alpha,\alpha+1})_*|\lambda|=\delta_{x_{r_q}},
\]
so \(|\lambda|\) is carried by the singleton fibre \(\{(x_{r_q},0)\}\).  The identity \((\pi_{\alpha,\alpha+1})_*\lambda=\delta_{x_{r_q}}\) then forces
\[
        \lambda=\delta_{(x_{r_q},0)}.
\]
Consequently every possible source lift is unique, and
\[
        \delta_{(x_{r_q},0)}\longrightarrow \delta_{(x,0)}
        \qquad(q\in Q^*,\ q\to\infty).
\]
Finally set
\[
        N_0(\alpha,\alpha)=\{r_q:q\in Q^*,\ q\text{ even}\},
        \qquad
        N_1(\alpha,\alpha)=\{r_q:q\in Q^*,\ q\text{ odd}\}.
\]
They are disjoint and both meet \(J(\alpha,\alpha,\alpha+1)\) infinitely, because the fibre \(Q^*\) meets the even and odd integers infinitely often.  This gives \((B)\) and \((C)\) for the new pair, and all remaining induction requirements are those of Talagrand's original successor step.
\end{proof}

Now define, as in Talagrand's construction,
\[
        H_q=Q_{n_q}\setminus \bigcup_{i,j<q}V_{m_{i,j}}(\theta_i,\alpha),
        \qquad
        Z_\alpha=\overline{\bigcup_{q\text{ even}}H_q},
        \qquad
        Y_\alpha=\overline{K_\alpha\setminus Z_\alpha}.
\]
Thus \(Y_\alpha\) and \(Z_\alpha\) are closed and \(Y_\alpha\cup Z_\alpha=K_\alpha\), so the successor space \(K_{\alpha+1}=(Y_\alpha\times\{0\})\cup(Z_\alpha\times\{1\})\) is compact.
Here \(\theta_0=\alpha\), and \((\theta_i)_{i\in I}\) enumerates \([\omega,\alpha]\) with \(0\in I\).

\begin{lemma}\label{lem:parity}
For each \(q\in Q^*\) the following hold:
\[
        y_{r_q}\in Z_\alpha\setminus Y_\alpha
        \quad\text{if }q\text{ is even},
\]
whereas
\[
        y_{r_q}\in Y_\alpha\setminus Z_\alpha
        \quad\text{if }q\text{ is odd}.
\]
Moreover \(x_{r_q}\in Y_\alpha\setminus Z_\alpha\) for every \(q\in Q^*\), and \(x\in Y_\alpha\setminus Z_\alpha\).
\end{lemma}

\begin{proof}
By definition \(r_q=n_q\) for \(q\in Q^*\).  The avoidance requirement \((g)\) gives \(y_{r_q}=y_{n_q}\in H_q\).  If \(q\) is even, then \(H_q\) is one of the clopen pieces used to define \(Z_\alpha\), so \(y_{r_q}\in \operatorname{int}Z_\alpha\), hence \(y_{r_q}\notin Y_\alpha\).  If \(q\) is odd, the set \(H_q\) is not used in the even union.  For every even \(s\), \(H_s\subseteq Q_{n_s}\), and \(Q_{n_s}\cap Q_{n_q}=\varnothing\); since \(Q_{n_q}\) is clopen, it also misses the closure of the even union.  Hence the clopen neighbourhood \(Q_{r_q}\) of \(y_{r_q}\) misses \(Z_\alpha\), so \(y_{r_q}\notin Z_\alpha\).

Finally, all \(H_q\)'s are contained in \(Q\), and \(Q\) is clopen and disjoint from \(P\).  Thus \(Z_\alpha\subseteq Q\), while \(x\) and all \(x_{r_q}\)'s lie in \(P\).  Hence they belong to \(Y_\alpha\setminus Z_\alpha\).
\end{proof}

Let
\[
        c_\alpha=\{t\in K_{\alpha+1}:t(\alpha)=1\}=Z_\alpha\times\{1\}.
\]
By Lemma \ref{lem:parity}, every lift of \(y_{r_q}\) to \(K_{\alpha+1}\) lies in \(c_\alpha\) if and only if \(q\) is even; every lift of \(x_{r_q}\) lies outside \(c_\alpha\).

\begin{lemma}\label{lem:source-filter}
Let \(T\) be any later continuation of the construction, and let \(\fa_\alpha\) be the filter on \(\omega\) generated by the sets
\[
        J(\alpha,\alpha,\beta)\setminus F,
        \qquad \beta>\alpha,\quad F\in[\omega]^{<\omega}.
\]
Equivalently, \(\fa_\alpha\) is generated by the sets \(J(\alpha,\alpha,\beta)\), \(\beta>\alpha\), together with the Fr\'echet filter.  Then \(\fa_\alpha\) is not an ultrafilter, and every member of \(\fa_\alpha\) meets both \(N_0(\alpha,\alpha)\) and \(N_1(\alpha,\alpha)\) infinitely.  Moreover, if \(A\subseteq T\) is clopen and \(\widehat x_n\in T\) satisfy \(\pi_\alpha(\widehat x_n)=x_n\), then the \(0\)-\(1\) sequence \((\indicator_A(\widehat x_n))_{n\in\omega}\) has an \(\fa_\alpha\)-limit.
\end{lemma}

\begin{proof}
The first assertion is exactly Talagrand's conditions \((B)\)--\((D)\).  Given finitely many generators, condition \((D)\) gives a later \(J(\alpha,\alpha,\gamma)\) which is contained in their intersection modulo a finite set.  Removing finitely many points does not destroy condition \((C)\), which says that every such later \(J\) meets both \(N_0(\alpha,\alpha)\) and \(N_1(\alpha,\alpha)\) infinitely.

For the second assertion, the clopen set \(A\) depends on some stage \(\beta>\alpha\).  The measures \(\delta_{\pi_\beta(\widehat x_n)}\) belong to \(B_n(\alpha,\alpha,\beta)\).  By condition \((F)\), their weak-star limit along \(J(\alpha,\alpha,\beta)\) exists.  Evaluating at the clopen trace of \(A\) gives convergence of \(\indicator_A(\widehat x_n)\) along that generator of \(\fa_\alpha\).  Since the sequence is \(0\)-\(1\)-valued, one of the sets \(\{n:\widehat x_n\in A\}\) and \(\{n:\widehat x_n\notin A\}\) belongs to \(\fa_\alpha\).
\end{proof}

We now prove the promised persistent anti-extension statement.

\begin{proposition}\label{prop:persistent}
There is no later continuation of the construction and no homeomorphism \(H:\widehat F\to\widehat G\) between closed subspaces of the final compactum \(T\) such that
\[
        \pi_\alpha[\widehat F]=F,
        \qquad
        \pi_\alpha[\widehat G]=G,
        \qquad
        \pi_\alpha H=h\pi_\alpha .
\]
Equivalently, the Boolean isomorphism \(\clop(K_\alpha)/I_G\cong \clop(K_\alpha)/I_F\) induced by \(h\) has no final extension to the traces over \(\widehat G\) and \(\widehat F\).
\end{proposition}

\begin{proof}
Assume, towards a contradiction, that such an \(H\) exists.  For every \(n\in\omega\), choose
\[
        \widehat x_n\in \widehat F,
        \qquad \pi_\alpha(\widehat x_n)=x_n,
\]
which is possible because \(\pi_\alpha[\widehat F]=F\).  In particular, this chooses the points \(\widehat x_{r_q}\) for \(q\in Q^*\).
Then
\[
        \pi_\alpha H(\widehat x_{r_q})=h(x_{r_q})=y_{r_q}.
\]
By Lemma \ref{lem:parity}, the point \(H(\widehat x_{r_q})\) belongs to the final pullback of \(c_\alpha\) if and only if \(q\) is even.

The set
\[
        H^{-1}(c_\alpha\cap \widehat G)
\]
is clopen in \(\widehat F\).  Since \(T\) is zero-dimensional and \(\widehat F\) is closed, this relative clopen set extends to a clopen set \(A\subseteq T\).  Therefore
\[
        \indicator_A(\widehat x_{r_q})=
        \begin{cases}
        1,& q\in Q^*,\ q\text{ even},\\
        0,& q\in Q^*,\ q\text{ odd}.
        \end{cases}
\]
Thus \(\indicator_A(\widehat x_n)\) is equal to \(1\) on \(N_0(\alpha,\alpha)\) and to \(0\) on \(N_1(\alpha,\alpha)\).  If this \(0\)-\(1\)-valued sequence had an \(\fa_\alpha\)-limit, then either its \(1\)-set or its \(0\)-set would belong to \(\fa_\alpha\).  But the \(1\)-set misses \(N_1(\alpha,\alpha)\), while the \(0\)-set misses \(N_0(\alpha,\alpha)\); every member of \(\fa_\alpha\) meets both of these sets infinitely.  Hence the sequence has no \(\fa_\alpha\)-limit, contradicting Lemma~\ref{lem:source-filter}.
\end{proof}

\section{The predicted construction}\label{sec:predicted}

We now assemble the local lemma into a full construction.

Assume \(\diamondsuit(S)\), where \(S\subseteq\omega_1\) is stationary and co-stationary.  Fix in advance an ``ordinary'' scheduler
\[
        \sigma:\omega_1\longrightarrow \omega_1\times\omega_1
\]
with the following two properties: if \(\sigma(\alpha)=(\theta,
\delta)\), then \(\theta\leqslant\alpha\) and \(\delta\leqslant\alpha\), and, for every pair \((\theta,
\delta)\in[\omega,\omega_1)\times\omega_1\), the set
\[
        \{\alpha\in\omega_1\setminus S:\alpha\geqslant\max\{\theta,\delta\},
        \ \sigma(\alpha)=(\theta,
\delta)\}
\]
is cofinal in \(\omega_1\).  Such a scheduler exists because \(\omega_1\setminus S\) has cardinality \(\omega_1\) and can be partitioned into \(\omega_1\) many cofinal sets.

The actual scheduling map \(\eta\) is defined recursively.  If \(\alpha\in S\) and the diamond guess at \(\alpha\) is a valid code for disjoint infinite closed sets \(F,G\subseteq K_\alpha\) and a homeomorphism \(h:F\to G\), we declare \(\eta(\alpha)=(\alpha,0)\) and use the diagonal successor step below.  At every other stage, including invalid diamond guesses, we put \(\eta(\alpha)=\sigma(\alpha)\) and perform Talagrand's ordinary successor step.  Thus a valid diamond stage may postpone one ordinary task, but no ordinary task is lost, because every ordinary pair is scheduled cofinally often on \(\omega_1\setminus S\) above both coordinates.

At a valid diamond stage \(\alpha\in S\), choose \(x,x_n,y,y_n,P,Q,P_n,Q_n,R_n\) as in Section~\ref{sec:anti-extension}.  Reserve the two measure indices \(0\) and \(\alpha\) in the enumeration at stage \(\alpha\), and set
\[
        U_n(\alpha,0)=P_n\cup Q_n,
        \qquad
        \nu_n(\alpha,0)=\delta_{h(x_n)},
        \qquad
        \nu_n(\alpha,\alpha)=\delta_{x_n}.
\]
All remaining admissible sequences on \(K_\alpha\) are placed into the unused indices of the enumeration \((\nu_n(\alpha,
\delta))_{\delta<\omega_1}\).  Since \(\diamondsuit\) implies CH and \(K_\alpha\) is metrisable, there are only \(\omega_1\) many such sequences.

The diamond code is decoded using only \(K_\alpha\), already constructed at the beginning of stage \(\alpha\).  If the code is invalid, or if one of the two closed sets is finite, no diagonal object is introduced and the stage is ordinary.

\begin{proposition}\label{prop:construction-valid}
The above recursion produces a realisation of Talagrand's construction satisfying conditions \((A)\)--\((G)\).  Hence its final compactum \(T\) satisfies Talagrand's conclusion: \(C(T)\) has the Grothendieck property and the weak-star compact unit ball of \(C(T)^*\) contains no copy of \(\beta\omega\).
\end{proposition}

\begin{proof}
At non-diagonal stages, and at invalid diamond stages, this is Talagrand's original successor construction.  At a valid diagonal stage, the target sequence \((\nu_n(\alpha,0))\) is the sequence treated by \(\eta(\alpha)=(\alpha,0)\), and the finite requirements \((a)\)--\((h)\) are satisfied exactly as in Section~\ref{sec:compatibility}.  Lemma~\ref{lem:paired-successor-step} checks the only new sets, namely the added source indices \(r_q\), the set \(J(\alpha,\alpha,\alpha+1)\), and the new pieces \(N_0(\alpha,\alpha),N_1(\alpha,\alpha)\).  The table in Section~\ref{sec:compatibility} records which of \((A)\)--\((G)\) are unchanged and which are affected.  Thus \((A)\)--\((G)\) hold at \(\alpha+1\).  Limit stages are inverse limits and are unchanged.

Finally, the ordinary scheduler \(\sigma\) treats every admissible clopen-supported signed measure sequence cofinally often on \(\omega_1\setminus S\), and the enumerations at each metrisable stage remain exhaustive.  Hence the proof of Talagrand's Theorem~4 applies to the final compactum.
\end{proof}

\begin{theorem}\label{thm:no-homeomorphic-disjoint}
For the final compactum \(T\) of the diamond construction, no two disjoint non-metrisable closed subspaces of \(T\) are homeomorphic.
\end{theorem}

\begin{proof}
Suppose that \(F,G\subseteq T\) are disjoint closed non-metrisable subspaces and that \(H:F\to G\) is a homeomorphism.  By Lemma \ref{lem:reflection}, there is a club \(C\subseteq\omega_1\) such that, for every \(\alpha\in C\), the projections \(F_\alpha\) and \(G_\alpha\) are disjoint, the map \(H\) reflects to a homeomorphism \(H_\alpha:F_\alpha\to G_\alpha\), and \(F_\alpha\) is infinite.  By Lemma \ref{lem:diamond-guess}, the diamond sequence guesses this reflected triple on a stationary set.  Choose \(\alpha\in S\cap C\) at which the guess is correct.  The stage \(\alpha\) construction then applies Proposition \ref{prop:persistent} to \(H_\alpha:F_\alpha\to G_\alpha\), and forbids any final extension of it.  This contradicts the existence of \(H\).
\end{proof}

\begin{corollary}\label{cor:no-2dcp}
The compactum \(T\) does not have 2DCP.
\end{corollary}

\begin{proof}
Assume that \((K_n)_{n\in\omega}\) witnesses 2DCP in \(T\).  No \(K_n\) is finite: if \(K_N\) were finite, then \(|K_N|\geqslant 2^m|K_{N+m}|\geqslant 2^m\) for every \(m\), a contradiction.  Since \(T\) contains no non-trivial convergent sequence by Proposition \ref{prop:construction-valid} and Talagrand's theorem, no infinite compact subspace \(K_n\) is metrisable.  In particular, \(K_1\) is non-metrisable.  But \(K_0\) contains two disjoint compact subsets homeomorphic to \(K_1\).  These are disjoint non-metrisable closed homeomorphic subspaces of \(T\), contradicting Theorem \ref{thm:no-homeomorphic-disjoint}.
\end{proof}

\begin{proof}[Proof of Theorems~\ref{thm:A} and~\ref{thm:B}]
Theorem~\ref{thm:A} is exactly Theorem~\ref{thm:no-homeomorphic-disjoint} together with Corollary~\ref{cor:no-2dcp}.  Theorem~\ref{thm:B} follows from Proposition~\ref{prop:construction-valid}, which preserves Talagrand's conditions \((A)\)--\((G)\) and hence Talagrand's original conclusion.
\end{proof}

\section{Classical pointwise sequence quotients}\label{sec:pointwise-quotients}

This section is independent of the diamond diagonalisation.  Its first result is specific to spaces \(C_{\ptop}(X)\), but is valid for every Tychonoff space \(X\).  The proof uses the finite-support description of \(C_{\ptop}(X)'\) and a functional-boundedness lemma of Banakh and Gabriyelyan.  Talagrand's compactum enters only afterwards, through the Grothendieck property of \(C(K)\) and the absence of a Banach quotient \(\ell_\infty\).

Following Banakh and Gabriyelyan, a locally convex space \(E\) has the \emph{Josefson--Nissenzweig property}, abbreviated JNP, if the identity map
\[
        (E',\sigma(E',E))\longrightarrow(E',\beta^*(E',E))
\]
is not sequentially continuous, where \(\beta^*(E',E)\) is the topology of uniform convergence on barrel-bounded subsets of \(E\) \cite[Definition~1.4]{BanakhGabriyelyan2023}.  For pointwise function spaces this agrees with the earlier definition of Banakh, the second-named author, and \'{S}liwa: \(C_{\ptop}(X)\) has JNP if and only if its dual contains a weak-star null sequence of finitely supported signed measures of norm one \cite[Corollary~3.10]{BanakhGabriyelyan2023}.  Moreover, this is equivalent to the existence of a quotient of \(C_{\ptop}(X)\) isomorphic to \((c_0)_{\ptop}\) \cite{BKS2019}.

\begin{proof}[Proof of Theorem~\ref{thm:cp-lp-quotient}]
Let
\[
        Q:C_{\ptop}(X)\longrightarrow(\ell_p)_{\ptop}
\]
be a continuous linear surjection.  For \(n\in\omega\), let \(e_n^*\) be the \(n\)-th coordinate functional on \(\ell_p\), and put
\[
        \mu_n=e_n^*\circ Q\in C_{\ptop}(X)'.
\]
Every \(\mu_n\) is a finitely supported signed measure on \(X\), and surjectivity of \(Q\) implies \(\mu_n\neq0\) for every \(n\).  For each \(f\in C(X)\), the vector \(Qf\) belongs to \(\ell_p\), and therefore
\[
        \mu_n(f)=(Qf)_n\longrightarrow0.
\]
Thus \((\mu_n)\) is weak-star null.  In particular, regarded as a subset of the dual of the Banach space \(C_b(X)\), the set \(M=\{\mu_n:n\in\omega\}\) is weak-star bounded.  Banakh and Gabriyelyan's support lemma implies that
\[
        S=\bigcup_{n\in\omega}\operatorname{supp}\mu_n
\]
is functionally bounded in \(X\) \cite[Lemma~3.4]{BanakhGabriyelyan2023}.

We claim that
\[
        \inf_{n\in\omega}\|\mu_n\|>0,
\]
where \(\|\mu_n\|\) is the total variation norm.  Otherwise, after passing to a strictly increasing sequence \((n_k)\), we could arrange
\[
        \|\mu_{n_k}\|\leqslant 2^{-2k}\qquad(k\geqslant1).
\]
Let \(y\in\ell_p\) be defined by
\[
        y_{n_k}=2^{-k},\qquad y_n=0\quad(n\notin\{n_k:k\geqslant1\}).
\]
By surjectivity, choose \(f\in C(X)\) with \(Qf=y\).  Since \(S\) is functionally bounded,
\[
        M_f:=\sup_{x\in S}|f(x)|<\infty.
\]
For every \(k\geqslant1\),
\[
        2^{-k}=|\mu_{n_k}(f)|
        \leqslant \|\mu_{n_k}\|M_f
        \leqslant M_f2^{-2k},
\]
which is impossible for sufficiently large \(k\).  This proves the claim.

Choose \(c>0\) with \(\|\mu_n\|\geqslant c\) for all \(n\), and set
\[
        \nu_n=\frac{\mu_n}{\|\mu_n\|}.
\]
Then \((\nu_n)\) is a weak-star null sequence of finitely supported signed measures of norm one.  By the characterisation quoted above, \(C_{\ptop}(X)\) has JNP and hence a quotient isomorphic to \((c_0)_{\ptop}\).
\end{proof}

\begin{remark}\label{rem:lp-cp-specific}
Theorem~\ref{thm:cp-lp-quotient} is a \(C_{\ptop}(X)\)-specific statement.  Its proof uses both the finite-support structure of \(C_{\ptop}(X)'\) and the functional boundedness of the union of the supports.  In particular, no assertion that JNP pulls back through arbitrary locally convex quotient maps is used.
\end{remark}

\begin{proposition}\label{prop:no-c0-linfty-pointwise}
Let \(K\) be a compactum satisfying Talagrand's Banach-space conclusion: \(C(K)\) is Grothendieck, and the weak-star compact unit ball \(M_1(K)\) contains no copy of \(\beta\omega\).  Then there is no continuous linear surjection
\[
        C_{\ptop}(K)\onto (c_0)_{\ptop}
\]
and no continuous linear surjection
\[
        C_{\ptop}(K)\onto (\ell_\infty)_{\ptop}.
\]
\end{proposition}

\begin{proof}
Suppose first that \(S:C_{\ptop}(K)\to (c_0)_{\ptop}\) is a continuous linear surjection.  Regard the same algebraic map as a map \(S:C(K)\to c_0\).  Its graph is closed for the Banach norms.  Indeed, if \(f_m\to f\) uniformly and \(Sf_m\to y\) in \(c_0\), then \(f_m\to f\) pointwise on \(K\), and by \(C_{\ptop}\)-continuity each coordinate of \(Sf_m\) converges to the corresponding coordinate of \(Sf\).  Norm convergence in \(c_0\) also implies coordinatewise convergence to \(y\), so \(y=Sf\).  The closed graph theorem makes \(S:C(K)\to c_0\) a bounded Banach-space surjection.  This is impossible because quotients of Grothendieck spaces are Grothendieck, whereas \(c_0\) is not Grothendieck.

The same closed graph argument applies to a continuous linear surjection \(C_{\ptop}(K)\to(\ell_\infty)_{\ptop}\).  It would give a Banach-space quotient \(Q:C(K)\onto\ell_\infty=C(\beta\omega)\).  Then \(Q^*:(\ell_\infty)^*\to M(K)\) is an isomorphic weak-star--weak-star embedding.  The Dirac measures \((\delta_p)_{p\in\beta\omega}\) form a weak-star copy of \(\beta\omega\) in a bounded ball of \( (\ell_\infty)^*\); after applying \(Q^*\) and rescaling, the unit ball \(M_1(K)\) would contain a weak-star copy of \(\beta\omega\), contradicting Talagrand's theorem.
\end{proof}

\begin{proof}[Proof of Corollary~\ref{cor:talagrand-classical-quotients}]
The cases \((c_0)_{\ptop}\) and \((\ell_\infty)_{\ptop}\) are Proposition~\ref{prop:no-c0-linfty-pointwise}.  If \(1\leqslant p<\infty\) and \(C_{\ptop}(K)\) had a quotient isomorphic to \((\ell_p)_{\ptop}\), Theorem~\ref{thm:cp-lp-quotient} would give a quotient \(C_{\ptop}(K)\to(c_0)_{\ptop}\), contradicting Proposition~\ref{prop:no-c0-linfty-pointwise}.
\end{proof}

Corollary~\ref{cor:talagrand-classical-quotients} is intentionally narrower than the full separable quotient problem for \(C_{\ptop}(K)\).  It rules out the classical coordinate quotients \((c_0)_{\ptop}\), \((\ell_p)_{\ptop}\) \((1\leqslant p<\infty)\), and \((\ell_\infty)_{\ptop}\).  It does not exclude an arbitrary infinite-dimensional metrisable quotient of \(C_{\ptop}(K)\).  In dual terms, the latter problem asks whether \(C_{\ptop}(K)'=\operatorname{span}\{\delta_x:x\in K\}\) contains an infinite-dimensional countable weak-star closed subspace.  The present argument does not decide that stronger question.

\section{Comments and open questions}

The proof shows more than the failure of 2DCP.  It shows that the diamond-guided Talagrand compactum is rigid against the first obstruction that any 2DCP witness would have to provide: two disjoint closed non-metrisable homeomorphic subspaces.  This is deliberately stronger than merely killing a specified binary tree of compacta.  The \(C_{\ptop}(X)\)-specific \(\ell_p\)-quotient principle in Section~\ref{sec:pointwise-quotients} is independent of Talagrand and holds for every Tychonoff space \(X\).  What is Talagrand-specific is the additional exclusion of \((c_0)_{\ptop}\) and \((\ell_\infty)_{\ptop}\), obtained by combining the closed graph theorem with Talagrand's original Banach-space conclusion.

The theorem does not assert that every realisation of Talagrand's construction fails 2DCP.  Nor does it rule out a different, symmetrically chosen realisation of a related construction having 2DCP.  The point is that Jensen's diamond supplies exactly the additional prediction missing from a bare CH bookkeeping: every final homeomorphism reflects to countable stages on a club, and the diamond sequence guesses one of those reflected traces stationarily often.  The persistent anti-extension lemma then kills the guessed trace at a single successor stage in a way which later coordinates cannot undo.

It is also worth contrasting this with the basic splitter obstruction of Lemma \ref{lem:splitter}.  The splitter alone prevents sequential or measure-convergent side branches, but it cannot by itself rule out 2DCP: Efimov-type spaces with 2DCP exist under additional axioms and, in some cases, in ZFC; see \cite{KKS2026}.  The new ingredient here is not a stronger splitter, but a Boolean anti-extension argument coupled to Talagrand's non-ultrafilter filters.

We leave three natural questions open.

\begin{problem}\label{prob:other-realisation}
Can another realisation of Talagrand's construction scheme have 2DCP?
\end{problem}

\begin{problem}\label{prob:ch}
Can the prediction principle in Theorem~\ref{thm:A} be weakened from \(\diamondsuit(S)\) to CH?
\end{problem}

\begin{problem}\label{prob:stronger-diagonalisation}
Can the diagonalisation be strengthened from homeomorphisms between disjoint closed non-metrisable subspaces to more flexible embeddings into disjoint closed sets, or even to all countable weak-star closed subspaces of \(C_{\ptop}(T)'=\operatorname{span}\{\delta_x:x\in T\}\)?
\end{problem}

A positive answer to Problem~\ref{prob:stronger-diagonalisation} would approach the full metrisable quotient problem for \(C_{\ptop}(T)\), which is not settled here.

\section*{Funding}
T.K. was supported by the Institute of Mathematics, Czech Academy of Sciences (RVO: 67985840).

\end{document}